\documentclass[aip,jcp]{revtex4-1}

\usepackage{amsmath,amssymb}
\usepackage{dcolumn}
\usepackage{bm}

\usepackage{amsthm}
\usepackage{mathrsfs}
\usepackage{graphicx}
\usepackage{hyperref}
\usepackage{xcolor}
\usepackage{fourier}
\usepackage{geometry}
\usepackage{extarrows}

\newcommand{\wh}[1]{\widehat{#1}}
\newcommand{\bl}{\big\langle}
\newcommand{\br}{\big\rangle}

\newtheorem{theorem}{Theorem}

\def\blue{\color{black}}
\def\Blue{\color{black}}
\def\red{\color{red}}

\def\ie{i.e.}
\def \blue{\color{black}}
\def \red{\color{black}}

\newcommand{\comment}[1]{}

\begin{document}

\title{The Derivation and Approximation of Coarse-grained Dynamics from Langevin Dynamics}
\author{Lina Ma}
\email{linama@psu.edu}
\affiliation{Department of Mathematics, the Pennsylvania State University, University Park, PA 16802-6400, USA.}%
\author{ Xiantao Li}
\email{xli@math.psu.edu}
\affiliation{Department of Mathematics, the Pennsylvania State University, University Park, PA 16802-6400, USA.}%
\author{Chun Liu}
\email{liu@math.psu.edu}
\affiliation{Department of Mathematics, the Pennsylvania State University, University Park, PA 16802-6400, USA.}%

\date{\today}

\begin{abstract}
We present a derivation of a coarse-grained description, in the form of a generalized Langevin equation, from the Langevin dynamics model that describes the dynamics of bio-molecules. The focus is placed on the form of the memory kernel function, {\blue the colored noise}, and the second fluctuation-dissipation theorem {\blue that connects them}. Also presented is a hierarchy of approximations for the memory and random noise terms, using rational approximations in the Laplace
domain. These approximations offer increasing accuracy. More importantly, they eliminate the need
to evaluate the integral associated with the memory term at each time step.   {\blue Direct sampling of the colored noise can also be avoided within this framework.} Therefore, the numerical implementation of the generalized Langevin equation is much more efficient.
\end{abstract}


\maketitle


\section{Introduction}
One of the most outstanding problems in molecular modeling of bio-molecular systems is the construction of coarse-grained (CG) models, in which {\blue only a few} degrees of freedom are explicitly retained. The importance of such development and { other various}  perspectives have been discussed in many review papers and books   \cite{Leach01,espanol2004statistical,riniker_developing_2012,noid_perspective:_2013,voth2008coarse}. {\blue These efforts have been driven by  the fact that direct simulations using an all-atom model are restricted by the small {\Blue time steps},  typically femto-seconds, while the time scale of interest is at least microseconds. Therefore, a CG model that allows one to efficiently explore events occurring on long time scales is critical to the understanding of molecular conformations and ultimately, biological functions.} The last two decades have witnessed a great deal of progress toward this goal \cite{noid_perspective:_2013,riniker_developing_2012,espanol2004statistical}. While many existing models have demonstrated their capability to recover (or predict) equilibrium properties, a systematic framework to incorporate dynamic properties is still challenging. In particular, as has been observed in the development of  MARTINI \cite{monticelli2008martini}, a remarkably successful coarse-grained force field, the effective friction mechanism is missing in such coarse-graining procedure. {\blue This observation has also been a strong motivation for the current work.}

A very important theoretical development in coarse-graining molecular models is the projection approach, originally formulated by Mori and Zwanzig \cite{Mori1965b,Zwanzig73}. This approach  is directly based on the dynamics of the full system, rather than the equilibrium statistical properties.
Such formalism (or similar reduction procedure) has recently been widely used to derive CG models based on the deterministic  Newton's equations of motion \cite{curtarolo_dynamics_2002,IzVo06,kauzlaric2011three,kauzlaric2011bottom,lange2006collective,Li2009c,oliva2000generalized,stepanova_dynamics_2007,LiXian2014,li2015incorporation}, known as molecular dynamics (MD) models.  The typical result is a {\blue \it generalized Langevin equation} (GLE), with a memory (or frictional) kernel  implicitly incorporating the influences of the {\blue degrees of freedom that have been projected out}. See the early works \cite{tully1980dynamics,guardia1985generalized,ciccotti1981derivation,AdDo76} for various derivations for interacting particle systems. Recently, there has been increasing interest in modeling complex dynamical systems using GLEs \cite{Li14,kauzlaric2011bottom,IzVo06,LiE07,fricks2009time,Darve_PNAS_2009,curtarolo_dynamics_2002,ChSt05}. The GLE is also driven by a stochastic force, which can be attributed to the uncertainty from the initial condition \cite{ChHaKu02}. Various schemes have been proposed to compute the memory and the random noise terms\cite{berkowitz1983generalized,berkowitz1981memory,hijon2006markovian,oliva2000generalized,LiXian2014,li2015incorporation}, but both of them are highly nontrivial due to the nonlocality of the kernel function in time. But in general, {\blue the practical issue of predicting correct dynamics properties} still remains. 

{\Blue In contrast to the large body of works for deterministic models, coarse-graining a stochastic system still remains a challenge}. As an initial attempt to treat a full molecular model that is {\it stochastic} in nature, we {\blue start with {\it Langevin dynamics}, which arises naturally from a molecular system in solvent \cite{Schlick2002}. {\blue The influence of solvent is not explicitly included. Rather it is modeled by a damping term and a white noise.  } In this sense, the Langevin dynamics model is {\it already} a coarse-grained model since the solvent particles have been removed. However, simulating the dynamics of a macromolecule using such a model is still challenging due to the number of atoms involved, and  the large intrinsic vibration frequencies, which requires small time steps.  Motivated by these facts,  we consider a further reduction, aiming to derive a model with even fewer variables. These CG variables could represent averaged atomic degrees of freedom, such as the center of mass of a cluster of atoms, or torsion angles. Typically, the time scale will be significantly improved, since the fast components that require small scale simulations are projected out. }

In  coarse-graining the Langevin dynamics model, treating the {\blue stochastic random term} usually requires {\blue special} considerations, compared to its deterministic counterpart. For instance,  the Mori-Zwanzig formalism \cite{Mori1965b,Zwanzig73} is not directly applicable to stochastic differential equations (SDE), since the semi-group operator is not available. We  {\blue therefore suggest an appropriate linearization, and then partition the full dynamics into subspaces. The variables associated with the subspace orthogonal to the CG variables are then eliminated by direct substitution.  }

The CG model that we {\blue have derived} is a slight generalization of the GLE, with an additional Markovian damping term. Further, we prove the {\blue \it fluctuation-dissipation theorem} (FDT) \cite{Kubo66} for this CG model, and the theorem takes a combined form of the first and second FDT (see  \cite{Kubo66}  for the distinction between these two FDTs). To our knowledge, such models have not been reported in the literature. In particular, we show that the memory kernel function depends on the damping coefficient in the full Langevin dynamics. {\blue Establishing} such direct connection is important for understanding the friction mechanism in the CG dynamics.

Although the new GLE model properly incorporates the influence of the degrees of freedom that have been removed,
the numerical implementation faces several challenges, as has been noted in many previous works \cite{berkowitz1983generalized,berkowitz1981memory,hijon2006markovian,oliva2000generalized}. {\blue In particular,  a direct solution procedure would involve the computation of a matrix function at every step, and the dimension of the matrix is almost the same as the dimension of the full system.} To alleviate the computational burden, we suggest an alternative computational approach, in which the kernel function is approximated by a rational function in the Laplace space. {\blue The  goal is to find an efficient approximation so that only a few parameters need to be calculated a priori. In this paper, we make use of the explicit formula for the memory kernel  to extract the numerical parameters in the approximate models.} This has several practical advantages. First, the approximate model in the time domain can be written as an extended system of SDEs, which are {\it memory-less}. As a result, no datum needs to be stored and no integral needs to be evaluated at each time step. This significantly reduces {\blue the computational cost since  numerical quadrature for the memory term is not needed}. Secondly, the random noise in the GLE can be approximated {\it indirectly} by introducing white Gaussian noises in the extended system. Therefore, there is no need to sample the random noise in the GLE, which otherwise requires non-trivial effort, e.g., Fourier-transform over long time period, or singular value decomposition of the covariance matrix \cite{li2015incorporation}.   We will provide the explicit forms of these approximations and illustrate how to determine the covariance of the noise to exactly satisfy the FDT.

{\blue The rational approximation is a novel, and yet quite flexible approach to model the memory effects by embedding a nonlocal model within a local one. In principle, there are various ways to determine the coefficients in the rational function. In this paper, we will test an idea of using the limiting values, both at zero and infinity, as interpolation points. But it is clear that this interpolation scheme may not be optimal: One may introduce other fitting procedures to obtain better accuracy. We leave this issue to future works.}

In general, evaluating the accuracy of a CG model with both mean force and damping coefficients is subtle,  partly because the mean force in the GLE needs to be parameterized and calibrated a priori. The error from that effort and the error from the rational approximation of the kernel function is difficult to separate in the present approach. Therefore, we consider a simple case in which the full model is linearized instead of a more complicated function form, with coefficients computed from a principal component analysis (PCA). This ensures that the covariance of the atomic coordinates and momenta are exactly captured. {\blue Starting with the harmonic model as the exact full model}, we are able to compute the memory function  explicitly. {\blue One implication is that the mean force is linear, which perhaps is the simplest and and most efficient coarse-grained force field, and  the model can be viewed as an elastic network model (ENM)\cite{delarue2002simplified,atilgan2001anisotropy}.  This model preserves the correct vibration models, and provides an ideal test problem for error assessment.  Further, the velocity auto-correlation can be computed analytically as well, allowing us to 
examine the accuracy without numerical and sampling errors.}

Perhaps the closest work to the present approach is the normal mode partition method for Langevin dynamics \cite{sweet2008normal}, in which the Langevin dynamics is projected to a subspace and its orthogonal complement space. Various truncation steps are taken to simplify the model. The method in \cite{sweet2008normal} is at the level of numerical algorithms.  What is presented in this paper also starts with such subspace partitions. However, instead of introducing a numerical algorithm at discrete time steps, we derive a CG model, and then introduce a systematic approximation procedure afterwards.

\section{Mathematical derivation}

\subsection{The full model}
We start with the full Langevin dynamics model with $N$ atoms, 
\begin{equation}\label{eq: ld0}
\left\{
 \begin{aligned}
   \dot{x}= & v,\\
    M \dot{v}=& F(x) - \Gamma v + f(t),   
 \end{aligned}\right.
\end{equation}
where $x=(x_1,x_2,\dots,x_N)$ denotes the coordinates of all the atoms, $M$ is a diagonal matrix  containing the mass of each atom,  $F(x)$ is the force from an empirical potentials $V(x):$ $F=-\nabla V$, $\Gamma$ denotes the damping coefficient for the friction term, and $f(t)$ is a stochastic force, usually modeled by a Gaussian white noise, which satisfies the fluctuation-dissipation theorem (FDT),
\begin{equation}\label{eq: fdt}
\bl f(t),f(t')^T\br =2k_BTM^{-1}\Gamma\delta(t-t').
\end{equation}
The FDT is crucial to ensure that the system reaches the correct equilibrium state. 

By introducing the following scaling, $$
\tilde{x} = M^{\frac{1}{2}}x,\quad \tilde{F} = M^{-\frac{1}{2}}F,\quad 
\tilde{\Gamma} = M^{-\frac{1}{2}}\Gamma M^{-\frac{1}{2}}, \quad\tilde{f} = M^{-\frac{1}{2}}f,
$$
we can remove the mass matrix from the system,  and work with a normalized system with unit mass for each atom. Rewriting every term without the tilde, the new system is expressed in the following form, 
\begin{equation}\label{scaleLGE}
\left\{
 \begin{aligned}
   \dot{x}= & v,\\
     \dot{v}=& F(x) - \Gamma v + f(t).   
 \end{aligned}
 \right.
\end{equation}
One can easily show that the new random noise $f(t)$ still obeys the FDT, \ie,
\begin{equation}
\bl f(t), f(t')^T \br=2k_B T\Gamma \delta(t-t').
\end{equation}

\subsection{A subspace partition}

There are many existing methods to implement the Langevin dynamics model \eqref{scaleLGE} numerically. However solving this system in its full form at every time step can be expensive, due to the large number of atoms, and the small time steps determined by the  stability condition of the numerical methods. In contrast, coarse-grained (CG) models that involve much fewer variables are more attractive. This will be the primary focus of this paper. To begin with, we let
$Y$ and $Y^\perp$ be two orthogonal subspaces, generated by basis functions $\Phi$ and $\Psi,$ respectively. Namely,
\begin{equation}
   Y= \text{Range}(\Phi), \quad Y^\perp= \text{Range}(\Psi).
\end{equation}
We assume the dimension of $Y$ to be $m$, where $m$ is much smaller than $3N$. The matrix $\Phi$ has  dimension $3N\times m$, and it will span  the space $Y$, the subspace generated by the CG variables. To ensure orthogonality, we  choose $\Psi$ with dimension $3N\times (3N-m)$ such that the following identities   holds, 
$$
\Phi^T\Psi=0, \quad \Phi^T \Phi = I_{m\times m}, \quad \Psi^T \Psi = I_{(3N-m)\times (3N-m)}.
$$

 We will project the Langevin equations into $Y$ and $Y^\perp$. For this purpose, we decompose the solution $x$ in the following form,
\begin{equation}\label{eq: qxi}
 x=\Phi q + \Psi \xi,
\end{equation}
where $q \in \mathbb{R}^m$ and $ \xi \in \mathbb{R}^{3N-m} $ are nodal values associated with the basis vectors in $\Phi$ and $\Psi$. There are many choices for the matrix $\Phi$. One choice can be the eigenvectors corresponding to low vibrational modes, {\blue as in the work of normal mode partition \cite{sweet2008normal},} or  the rotation and translation blocks (RTB) that describe the rigid-body motions. Here we will first focus on the  general framework and postpone the specific choices to later discussions.

Our next step is to rewrite the first-order system \eqref{scaleLGE} into a second-order equation, together with the decomposition \eqref{eq: qxi},  
$$
\Phi \ddot{q} +  \Psi \ddot{\xi} = F(\Phi q + \Psi \xi) - \Gamma \Phi \dot{q} - \Gamma \Psi \dot{\xi}+f(t).
$$
Our goal is to eliminate $\xi$, which represents the degrees of freedom associated with $Y^\perp$.

To proceed, by left multiplying both sides by $\Phi^T$, we turn the equation above into,
\begin{equation}\label{eq: q}
\ddot{q}  = \Phi^T F(\Phi q + \Psi \xi) - \Phi^T \Gamma \Phi \dot{q} - \Phi^T \Gamma \Psi \dot{\xi}+\Phi^T f(t).
\end{equation}

Notice  $\xi$ disappeared from the left hand side thanks to the orthogonality condition.  
Similarly, we left multiply both sides by $\Psi^T$, and arrive at a second order differential equation for $\xi$, 
\begin{equation}\label{eq: xi}
 \ddot{\xi} = \Psi^T F(\Phi q + \Psi \xi) - \Psi^T \Gamma \Phi \dot{q} - \Psi^T \Gamma \Psi \dot{\xi}+\Psi^T f(t).
\end{equation}

The nodal values $q$ and $\xi$ are still coupled together for now.
To eliminate the variable $\xi$, we solve equation \eqref{eq: xi} analytically, together with a subsequent substitution into
\eqref{eq: q}. This is clearly intractable due to the nonlinearity of $F$ in \eqref{eq: xi}. Therefore,  we simplify the derivation by using a linearization $F=-Ax$. Since the matrix $A$ can be related to the covariance of the atomic coordinates, it can be determined from the principal component analysis (PCA). Namely we have
$$
\bl x(t), x(t)^T \br = k_BT A^{-1},
$$
which can be directly computed from a molecular simulation. Such  linear approximation is a necessary route in many 
coarse-graining procedures \cite{IzVo06,Li2009c,oliva2000generalized,stepanova_dynamics_2007}. In principle, higher order expansions  can  be introduced, e.g., \cite{ChSt05,WuLiLi2015}, but the derivation is exceedingly complicated. 

Recall that the basis functions are normalized, i.e., $\Phi^T\Phi=I$ and $\Psi^T\Psi=I$. We further define the following terms,
$$ A_{11} = \Phi^T A \Phi, \quad A_{12} = \Phi^T A \Psi, \quad \Gamma_{11} = \Phi^T\Gamma\Phi, \quad \Gamma_{12} = \Phi^T \Gamma \Psi,  {\blue \quad f_1=\Phi^T f(t),}
$$
$$
 A_{21}=\Psi^T A \Phi, \quad A_{22}=\Psi^TA \Psi, \quad \Gamma_{21} = \Psi^T\Gamma \Phi, \quad \Gamma_{22}=\Psi^T\Gamma\Psi, {\blue \quad f_2=\Psi^T f(t)}.
$$

Since it is usually easier to work with first order systems, we now convert the higher order equations back to a coupled first order system, with $p$ representing the momentum  of $q$,  and $ \eta$ being the momentum of  $\xi$. Notice  that we are dealing with unit mass system, the momentum also represents velocity. The first order system reads
\begin{equation}\label{eq: lg'}
\left\{
 \begin{aligned}
   \dot{q}= &  p, \\
   \dot{p}=&\Phi^T F(\Phi q) - A_{12} \xi  - \Gamma_{11} p - \Gamma_{12}  \eta +{\blue f_1(t)},\\
   \dot{\xi}= &  \eta, \\   
   \dot{\eta}= & -A_{21} q  - A_{22}\xi - \Gamma_{21} p - \Gamma_{22}  \eta + {\blue f_2(t)}.
 \end{aligned}
 \right.
\end{equation}
So far,  we have not done any dimension reduction yet, and these equations are equivalent to the original dynamics within the linear approximation. {\blue In addition, the random forces $f_1(t)$ and $f_2(t)$ are projections of the original white noise. Since $f_1(t)$ is directly influencing the CG variable $p$, it will be retained in the CG model. On the other hand, the influence of $f_2$ on the CG variables will be revealed by the coarse-graining procedure. } 
\subsection{The reduction of the number of degrees of freedom}

We take $(q,p)$  as the quantities of interest, i.e., the CG variables, solve $\xi, \eta$ explicitly, substitute them back to the equations for ($q, p$),  and thus eliminate ($\xi$,$\eta$) in the system \eqref{eq: lg'}. Detailed computations are shown in Appendix \ref{memdef} due to the lengthy calculations.  The CG equations for $(q, p)$ are then given by,  
\begin{equation}\label{eq: GLE}
\left\{
\begin{aligned}
 \dot{q}= &  p, \\
   \dot{p}=& F_\text{eff}(q) - \Gamma_{11} p - \int_0^t \theta(t-\tau) p(\tau)d\tau + \widehat{f}.
\end{aligned}
\right.
\end{equation}
Here $F_\text{eff}(q) $ is an effective force field for the CG variables.

We  refer to $z=\int_0^t \theta(t-s) p(s)ds$ as the memory term, and $\theta(t)$ as the {\it memory kernel function}, which is expressed in terms of a matrix exponential \cite{golub2012matrix}, 

\begin{equation}\label{eq: theta}
 \theta(t)= 
 \big[ A_{12}, \;\Gamma_{12}\big] e^{-Gt} 
 \left[
\begin{array}{cc}
 A_{22}^{-1} &  0\\
 0 & -I
 \end{array}
 \right]
 \left[
\begin{array}{c}
  A_{21}  \\
  \Gamma_{21}
 \end{array}
\right],
\end{equation}
{\blue where the matrix $G \in \mathbb{R}^{(6N-2m)\times (6N-2m)}$  is defined as,}
\begin{equation}\label{eq: Gmat}
G=
\left[
\begin{array}{cc}
 0 & -I\\
 A_{22} &\;\;\;\; \Gamma_{22} 
 \end{array}
 \right] .
\end{equation}

What complicates the derivation from \eqref{eq: lg'} is the presence of the stochastic noise $f_2(t)$ in the last equation. 
With lengthy calculations, we have shown that $\widehat f(t)$ is a combined Gaussian random noise. It is a stationary Gaussian random process with mean zero, satisfying the second fluctuation-dissipation theorem:
\begin{equation}\label{eq: fdt3}
 \bl \widehat{f}(t) \widehat{f}(t')^T \br = 2k_B T {\blue \Gamma_{11}} \delta(t-t') + k_B T \theta(t-t').
\end{equation}
Interestingly, this takes a combined from of the first and second FDT. The proof of this FDT is provided in the Appendix \ref{fluc}. Stationary Gaussian processes with mean zero are uniquely determined by their time correlation functions  \cite{Doob44}. Therefore, the CG model \eqref{eq: GLE} is closed once the memory kernel is known.

\section{Further Approximation of the Generalized Langevin Equations}
Solving the GLEs \eqref{eq: GLE} directly is clearly not practical: On one hand, one needs to keep
the history of the solution in order to compute the memory term; on the other hand, evaluating the memory function
at each time step is very expensive due to the large dimensionality of the matrix $G$ (the size is $(6N-2m)\times (6N-2m)$). {\blue For example, direct evaluation of $\theta(t)$ involves the computation of the matrix exponential $\exp(-Gt)$, which in general is quite expensive \cite{golub2012matrix}.}  It is thus natural to develop algorithms to approximate the memory term to make the CG model \eqref{eq: GLE} easier to implement, and become
 truly useful in practice.

{\blue In order to approximate the memory term, we propose a general approximation method, which will address these issues under the same framework. Rather than targeting the time-domain values of the memory kernel directly, we will work with its Laplace transform. The coefficients in our approximation, which only need to be computed once, can be determined by fitting or interpolation a priori.  As an example, we first present an interpolation procedure similar to the standard Hermite interpolation in numerical analysis.}
 This interpolation
 requires the following terms: \( \int_0^\infty \theta(t)dt, \quad \theta(0), \quad \theta'(0)\) etc, which can all be computed from the explicit expression of $\theta(t)$ \eqref{eq: theta}.  In particular, we define the {\it moments},
\begin{equation}\label{eq: moms}
 M_0= \theta(0), \;\;M_1= \theta'(0), \;\;M_2= \frac{1}{2!} \theta''(0),\;\; \cdots, \;M_\ell= \frac{1}{\ell!} \theta^{(\ell)}(0),\;\cdots, \;\;M_\infty=\int_0^\infty \theta(t)dt.
\end{equation}

The first approximation is to replace the memory function by a delta function, i.e.,
\begin{equation}
 \theta(t) \approx M_\infty \delta(t),
\end{equation}
which  leads to the approximation of memory term,
\begin{equation}
z\approx M_\infty {p}(t).
\end{equation}
Clearly this results in an added damping to the dynamics. Therefore, we will define $\Gamma_{\text{add}}= M_\infty$. This simple selection ensures that $\int_0^\infty \theta(t)dt$ is preserved, mimicking a Green-Kubo type of formula. It predicts the correct long-time behavior of the dynamics. The resulting model is still a Langevin dynamics model, which 
will be referred to as the {\bf zeroth-order approximation}. {\blue At the same time, we still need to ensure that the second FDT \eqref{eq: fdt3}
  is satisfied here. Therefore we need to add an appropriate noise such that equation (\ref{eq: fdt3}) holds. }  More specifically, we have,
\begin{equation}
\left\{
\begin{aligned}
 \dot{q}= &  p, \\
   \dot{p}=&  F_\text{eff}(q) - (\Gamma_{11}+ \Gamma_{\text{add}}) p + {\blue \widehat{f}}.
\end{aligned}
\right.
\end{equation}
Here $\blue \widehat{f}$ is a white noise, satisfying the FDT,
\begin{equation}
 \bl {\blue \widehat{f}(t) \widehat{f}(t')}^T \br = 2k_B T \big(\Gamma_{11}  + \Gamma_{\text{add}} \big) \delta(t-t').
\end{equation}
Since  the zeroth  order approximation  is in the same form as the Langevin dynamics \eqref{scaleLGE}, the implementation is straightforward. Many methods are available \cite{helfand1979numerical,leimkuhler2011comparing,skeel1999integration,van1982algorithms,wang2003analysis}. One only needs to change a few parameters in the numerical scheme and work with a much smaller   {\blue number of variables}. 

{\blue In light of the second FDT \eqref{eq: fdt3}, we observe that $M_\infty$ is analogous to the correlation time, and therefore represents important time scales. Known as the Markovian approximation, the approximation  by $M_\infty$ has been used in many other works \cite{kauzlaric2011bottom,hijon2006markovian,hijon2010mori}, and as observed in many numerical tests, the approximation is only satisfactory where there is significant time scale separation. But in general it is inadequate. Next we will present higher order approximations.}

\bigskip

Our general approximation scheme is based on the Laplace transform of $\theta(t)$, defined as,
\begin{equation}\label{eq: Th}
 \Theta(\lambda)=\int_0^{+\infty}\theta(t) e^{-t/\lambda} dt.
\end{equation}
Notice that we have chosen to work with the variable $\lambda$ (which has unit of time), instead of the usual choice $s$ ($s=1/\lambda$).  As an example, we approximate the Laplace transform of $\theta$ by a rational function,
\begin{equation}
 \Theta(\lambda) \approx R_{1,1}(\lambda), \quad  R_{1,1}(\lambda) \overset{\text{def}}{=} [I - \lambda B]^{-1} C\lambda,
\end{equation}
and the matrices $A$ and $B \in \mathbb{R}^{m\times m}$ are to be determined. More specifically, we enforce the following two conditions:
\begin{equation}\label{eq: mom1}
  \begin{aligned}
    \Theta'(0)=& R_{1,1}'(0),\\
    \Theta(+\infty)=& R_{1,1}(+\infty).
  \end{aligned}
\end{equation}

Direct calculations yield, 
\begin{equation}
  \Theta'(0)= M_0, \;\;  \Theta(+\infty)=M_\infty.
\end{equation}
After solving the equations \eqref{eq: mom1}, we find that, 
\begin{equation}\label{eq: AB}
 \begin{aligned}
  C=&M_0,\\
  B=&-M_0 M_\infty^{-1}.
  \end{aligned} 
\end{equation}

With this rational approximation, the memory term satisfies an additional equation, 
\begin{equation}
 \dot z= Bz + C{p} + \zeta,
\end{equation}
and $\zeta(t)$ is an added white noise, which will facilitate the approximation of the colored noise $\wh{f}$ in the GLE \eqref{eq: GLE}. {\blue This is motivated by the fact that the Gaussian noise $\wh{f}$ in the GLE is correlated in time. We construct a colored Gaussian noise through a first order stochastic differential equation. The resulting stochastic force is an Ornstein-Uhlenbeck process. Such an approximation scheme is known as Markovian embedding \cite{bao2005non}.} It effectively eliminates the need to sample the colored noise $\wh{f}$ directly.

This amounts to an approximate model, which will be referred to as the {\bf first-order approximation}, given by,
\begin{equation}\label{eq: 1st-order}
\left\{
\begin{aligned}
 \dot{q}= & p, \\
   \dot{p}=&F_\text{eff}(q)- \Gamma_{11} p - z + {\blue f_1},\\
   \dot{z}=& Bz + C{p} + {\blue  \zeta(t)}.
\end{aligned}
\right.
\end{equation}

It is not yet clear how the memory kernel and the random noise $\wh{f} $ are approximated in the time domain, and more importantly, whether they still satisfy the FDT \eqref{eq: fdt3}. 
To demonstrate that this procedure indeed leads to a consistent approximation of the memory term $z(t)$ and the noise  $\wh{f}$, we formulated the following theorem. In particular, we provide a simple formula for the covariance of the additive noise $\zeta(t).$
\begin{theorem}
  Assume that $z(0)$ is a Gaussian random variable with mean zero and covariance $k_B T C$ given by \eqref{eq: AB}. Further,  assume that 
  the noise $\zeta(t)$ has covariance $\Sigma,$ given by, 
  \begin{equation}\label{eq: sigma}
  \Sigma= -2k_B T CB.
\end{equation}
  Then the first-order model \eqref{eq: 1st-order} is equivalent to an approximation of the GLE, in which the
  memory function is approximated by, 
    \begin{equation}\label{eq: mem1}
 \theta(t)\approx e^{Bt} C, 
\end{equation}
and the second FDT \eqref{eq: fdt3} is {\bf exactly} preserved. 
\end{theorem}

\begin{proof}
 This demonstrates how the memory function and random noises in the GLEs are consistently approximated in this approach. To show the equation \eqref{eq: sigma}, we let the covariance of $z(0)$ be $C$, and the covariance of  $\zeta(t)$  be $\Sigma$, which is to be determined.  It is clear that we can write the solution of the last equation {\blue of system (\ref{eq: 1st-order})} as follows using the variation of constant formula,
\begin{equation}
  z(t)= e^{Bt}z(0) + \int_0^t e^{B(t-\tau)} \zeta(\tau) d\tau+ \int_0^t e^{B(t-\tau)} {\blue C} p(\tau) d\tau.
\end{equation}
Therefore, the memory term is approximated by the third term with the kernel function approximated by \eqref{eq: mem1}. 

Meanwhile, the first two terms make a stationary Gaussian process, denoted by $g(t),$ if the following Lyapunov equation  \cite{risken1984fokker} holds,{\blue
\begin{equation}
  k_B T \big[ CB^T+BC \big]=-\Sigma.
\end{equation}
Observe that $B$ and $C$ are determine from \eqref{eq: AB}. In particular, we have $BC=CB^T.$} Thus a simple substitution leads to \eqref{eq: sigma}.

{\blue Finally, a substitution of $z(t)$ into the second equation in \eqref{eq: 1st-order} shows that the random process $g(t)$ will become an approximation of $\wh{f}(t).$ With direct calculations, we find that,
\[ \big\langle g(t) g(t')^T \big\rangle = k_B T e^{(t-t')B} C,\]
for any $t'\le t.$ In light of \eqref{eq: mem1}, we find that the approximate kernel function and the approximate random noise still satisfy the second FDT \eqref{eq: fdt3}.}
\end{proof}

\bigskip

This approach can be easily extended to higher order. For example, we can choose a rational function as follows,

\begin{equation}\label{eq: r22}
 \Theta(\lambda) \approx R_{2,2}(\lambda), \quad  R_{2,2}(\lambda) \overset{\text{def}}{=} [I - \lambda B_0 - \lambda^2 B_1 ]^{-1} [\lambda C_0+\lambda^2C_1].
\end{equation}
{\blue The parameters $B_0, B_1, C_0$ and $C_1$ will be determined from an interpolation procedure. We will adopt the conventional Pad\'e approximation, and expand both $\Theta$ and $R_{2,2}$ around $\lambda=0.$ Also known as moment matching, the Pad\'e approximation will ensure that the first few coefficients match with those in the rational function. This leads to the following matching conditions, referred to as {\it moment equations},}
\begin{equation}\label{eq: mom2}
  \begin{aligned}
    \Theta'(0)=& R_{2,2}'(0),\\
    \Theta''(0)=& R_{2,2}''(0),\\
    \Theta'''(0)=& R_{2,2}'''(0),\\
    \Theta(+\infty)=& R_{2,2}(+\infty).
  \end{aligned}
\end{equation}
{\blue This last condition, which is not from the standard Pad\'e approximation, is enforced here to incorporate the limit as $\lambda \to +\infty.$}

 With direct calculations, 
we have the equations for the coefficients,
\begin{equation}\label{eq: m2}
 \begin{aligned}
  C_0=&M_0,\\
  C_1 + B_0 C_0=&M_1,\\
  B_0 M_1  + B_1 C_0 =& M_2,\\
  C_1=&-B_1 M_\infty.
  \end{aligned} 
\end{equation}
Here, the moments $M_i$s have been defined in \eqref{eq: moms}.

By substituting the first and last equations into the second and third equations, we can simplify the equations into a 2-by-2 block system,
\[ 
\begin{aligned}
-B_1 M_\infty + B_0 M_0= &M_1, \\
 B_0M_1 + B_1 M_0= &M_2,
 \end{aligned}\]
from which the coefficients $B_0$ and $B_1$ can be determined. Then $C_0 $ and $C_1$ are immediately available from the first and last equations in \eqref{eq: m2}.  

As in the first order approximation, we can also eliminate the memory by introducing auxiliary variables that satisfy  additional equations. {\blue To see this, we start with the memory term $z$ and with the second order rational approximation, we have,
\[ s^2 Z - s B_0 Z - B_1 Z = s C_0 P + C_1 P,\]
 where $Z$ and $P$ are respectively the Laplace transform of $z(t)$ and $p(t).$ In order to convert this equation to the time domain, we need the initial values for $z(t)$. In particular, we have $z(0)=0,$ and by direct differentiations, we have $\dot{z}(0)= \theta(0) p(0).$

Next using the fact that the Laplace transform of $\dot z$ is given by $sZ-z(0)$ and the Laplace transform of $\ddot z$ is given by $s^2 Z -s z(0) - \dot z(0),$ we can convert this equation to the time domain,
\[ \ddot z - B_0 \dot z - B_1 z = C_0 \dot p + C_1 p,\]
provided that $C_0=\theta(0), $ which is exactly the first matching condition in \eqref{eq: m2}. }
 
We can write this second order equation into a first order form, by introducing another variable $z_1$: $z_1=\dot z - B_0 z$. They satisfy the following  differential equations,
\begin{equation}\label{eq: 2nd-order}
\left\{
\begin{aligned}
 \dot{q}= & p, \\
   \dot{p}=&F_\text{eff}(q)- \Gamma_{11} p - z +{\blue f_1},\\
   \dot{z}_1=& B_1 z + C_1 p + {\red \zeta_1(t)},\\
   \dot{z}=& z_1 + B_0 z + C_0 p + \zeta(t).
\end{aligned}
\right.
\end{equation}
{\blue Again, we have added a white noise $\zeta(t)$ and $\zeta_1(t)$ to each additional equation, which will lead to an approximation of the colored noise $\wh{f}(t)$ in the exact CG model \eqref{eq: GLE}.}

We would like to point out that approximating the memory kernel using exponential functions has been used in \cite{baczewski2013numerical}, where the memory function is approximated by a sum of exponential functions for the case when the dimension of $q$ is 1. Known as Prony sum, such a method is very useful in approximating convolutional integrals \cite{ou2014reconstruction,jiang2004fast,arnold2003discrete}.  On the other hand, our ansatz is more general, and it is suitable for matrix-valued kernel functions.

The corresponding approximation will be referred to as the {\bf second-order approximation}. In the Appendix \ref{sec: 2nd}, we have shown
how to choose the initial conditions for $z$ and $z_1$, along with the covariance for $\zeta(t)$ and $\zeta_1(t)$, so that the approximation of the
memory and random noise terms are consistent, in the same spirit as Theorem 1. The result can be summarized as the second theorem,
\begin{theorem}
  Assume that $z(0)$ and $z_1(0)$ are Gaussian random variables with mean zero and appropriate covariance. 
  Then the second order model \eqref{eq: 2nd-order} is equivalent to an approximation of the GLE \eqref{eq: GLE}, in which the approximations of the memory kernel and the random noise are consistent in the sense that the second FDT \eqref{eq: fdt3} is {\bf exactly} preserved. 
\end{theorem}
The proof of this theorem is provided in the Appendix D.

\medskip
{\blue From the first and second order approximations, one can already see the advantages of the rational approximation in terms of the Laplace transform. On one hand, the memory kernel in the original GLE does not need to be computed at every step. The memory effect, however, is not neglected. Rather, it is incorporated via an extended system. Clearly, solving a few additional linear differential equations is much more efficient than computing an integral at every time step. {\Blue For example, a direct solution method would involve computing the memory term at every step. At the $n$th step, this would require $n$ matrix-vector multiplications to collect terms from all previous time steps. If the total number of time steps is $N$, then the number of such operations would be about $N^2/2$. In contrast, the implementation of the model \eqref{eq: 2nd-order} would only require about $4N$ such operations in total. Of course, in a direct method, computing the memory function at each step also adds to the computational cost. }   On the other hand, the random noise $\wh{f}(t)$ is approximated by a colored noise, generated from the {\it same} extended system by just adding a white noise to each additional equation. This way, we avoid the problem of sampling the correlated noise  $\wh{f}(t)$, which in practice, can be highly nontrivial.}

Finally, we present the third-order approximation, i.e., 
$$
 \Theta(\lambda) \approx R_{3,3}(\lambda), \quad  R_{3,3}(\lambda) \overset{\text{def}}{=} \big[I - \lambda B_0 - \lambda^2 B_1 - \lambda^3 B_2\big]^{-1} \big[\lambda C_0+\lambda^2C_1+\lambda^3 C_2\big].
$$
Similarly to the second order approximation, we only need to match the limiting values as $\lambda \to 0,$ and $\lambda \to +\infty$. More specifically, we write the rational function as,
\begin{equation}
R_{3,3} \sim \lambda M_0 + \lambda^2 M_1 + \lambda^3 M_2 + \lambda^4 M_3 +\cdots \lambda^5 M_4 + \cdots,
\end{equation}
and we enforce the first five moments to match those of the exact kernel function. As a result, one can proceed as follows,
 \begin{equation}
\lambda {\blue C}_0+\lambda^2 {\blue C}_1+\lambda^3 {\blue C}_2 \sim \big[I - \lambda B_0 - \lambda^2 B_1 - \lambda^3 B_2\big] \big[ \lambda M_0 + \lambda^2 M_1 + \cdots \lambda^5 M_4 + \cdots\big].
\end{equation}
Matching the first five moments, one arrives at,
\begin{align}\label{eq: 3rdordercoeff}
\begin{split}
  C_0=&M_0,\\
  C_1 + B_0 C_0=& M_1,\\
    B_1 M_0 + C_2 =& M_2,\\
   B_0 M_2+{\red B_1M_1}+B_2M_0 =& M_3,\\
   {\red B_0 M_3+B_1M_2+B_2M_1} = &M_4,\\
  C_2=&-B_2M_\infty.
  \end{split}
\end{align}
Again the last equation comes from matching the moment $M_\infty.$

By directly substituting  the first and last equations into the third equation, one obtains a complete set of  linear equations for $B_0$, $B_1$ and $B_2$ (equations 3-5 in (\ref{eq: 3rdordercoeff})). Then, the remaining coefficients can be determined directly from the remaining three equations.  This procedure for solving the coefficients $B_i$s seems to be general.

We can continue to approximations of arbitrary order. The matching procedure involves the values of $\Theta$ which are provided here,
\begin{equation}
 M_\infty =  \big[ A_{12}, \;\Gamma_{12}\big] G^{-1}
 \left[
\begin{array}{cc}
 A_{22}^{-1} &  0\\
 0 & -I
 \end{array}
 \right]
 \left[
\begin{array}{c}
  A_{21}  \\
  \Gamma_{21}
 \end{array}
\right]. 
\end{equation}
and,
\begin{equation}
M_\ell=\frac{1}{\ell!} \Theta^{\ell}(0) =  (-1)^\ell \big[ A_{12}, \;\Gamma_{12}\big] G^{\ell}
 \left[
\begin{array}{cc}
 A_{22}^{-1} &  0\\
 0 & -I
 \end{array}
 \right]
 \left[
\begin{array}{c}
  A_{21}  \\
  \Gamma_{21}
 \end{array}
\right],
\end{equation}
for all $\ell \ge 1$.


\section{Numerical Results}

To test the effectiveness of the approximate models, several numerical tests have been conducted. As alluded to in the introduction,
we linearized the dynamics with the matrix $A$ determined from the PCA analysis. As a specific example, we consider the protein Chignolin
(PDB id 1uao, see Figure \ref{1uao}), which is a peptide with 10 {\blue residues}, amino acids bonded together by peptide bonds. Simulations have been run in TINKER \cite{ponder2004tinker} at temperature $T=298$ for .4 ns with time step 1fs. The system is set up in solvation, modeled by the generalized Born (GB) model. We then use the data upon equilibrium, and compute the matrix $A=k_B T \bl x,x^T\br^{-1}$. Two separate runs have been conducted {\blue with constant damping coefficients} {\bf (a)} $\Gamma=91 ps^{-1}$ and {\bf (b)} $\Gamma=5 ps^{-1}.$ They model respectively a high friction and a low friction case. 
The CG variables are defined using the rotational and translational blocks (RTB), which is a useful way to capture the low vibrational modes \cite{licu02,TaGaMaSa00}. More specifically, each residue is regarded as a rigid body and represented by six degrees of freedom, including three translational and three rotational modes. 
{\blue For our model system, the full model $x$ has dimension $N=414$ (three physical dimension for each particle). The CG variable $q$ has dimension $m=60$, with 6 dimensions for each residual.} 
 We comment that the RTB blocks have also been used to derive CG models, e.g., in 
\cite{essiz2006rigid}. But in \cite{essiz2006rigid} the memory effect has been ignored. 
\begin{figure}[htp]
\includegraphics[scale=0.3]{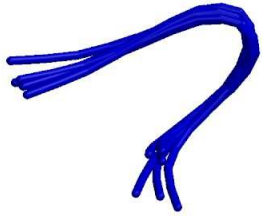}
\caption{A Cartoon view of the protein Chignolin (PDB Id: 1uao)}\label{1uao}
\end{figure}

We choose the velocity auto-correlation as a target dynamic quantity, to test the accuracy of our approximate models.  Due to the linearity, the velocity auto-correlation function can be expressed explicitly using matrix exponential. The derivation is given in Appendix \ref{auto}. The correlation function from the full dynamics is regarded as the exact result. For the approximate models, we have also derived the formulas for the auto-correlations expressed in terms of matrix exponential again, as shown in appendix \ref{auto}. All the matrix exponentials are computed in MATLAB using its built-in function {\it expm}. 

First, we compare the approximate memory functions from the first, second and third order approximations to the exact memory kernel given by \eqref{eq: theta}.  {\blue Since $\theta(t)$ is a matrix-valued function, we pick out the {\red sixth} diagonal entry of the matrix and evaluate it for the time period $t\in [0, 0.1].$ This corresponds to {\red the last rotational component} of the first residue.}  As shown in {\blue Figure \ref{fig: memker}},  our hierarchy of approximations offer increasing accuracy in the approximation of the kernel function in both cases (high friction case $\Gamma=91 ps^{-1}$ and low friction case $\Gamma=5 ps^{-1}$). {\red In the high friction case, we can observe improvement as the approximation order gets higher, and the third order approximation is the most satisfactory. In the low friction case, the kernel function is quite oscillatory. In this case, the first order approximation is not acceptable at all. The second and third order approximations show very good agreement, but only up to $t = 0.012$, and the fourth order model predicts the kernel well in a larger interval, up to $t = 0.018.$ The fourth order approximation is included here to show that the approximations still have improving accuracy.} 
This can be attributed to the fact that the moments are related to the derivatives of $\theta(t)$ at $t=0,$ and as more moments are incorporated, the accuracy of the approximation can be guaranteed for a longer period of time. The zeroth-order approximation is not shown here since it is a delta function.

\begin{figure}[htp]
\includegraphics[scale=.4]{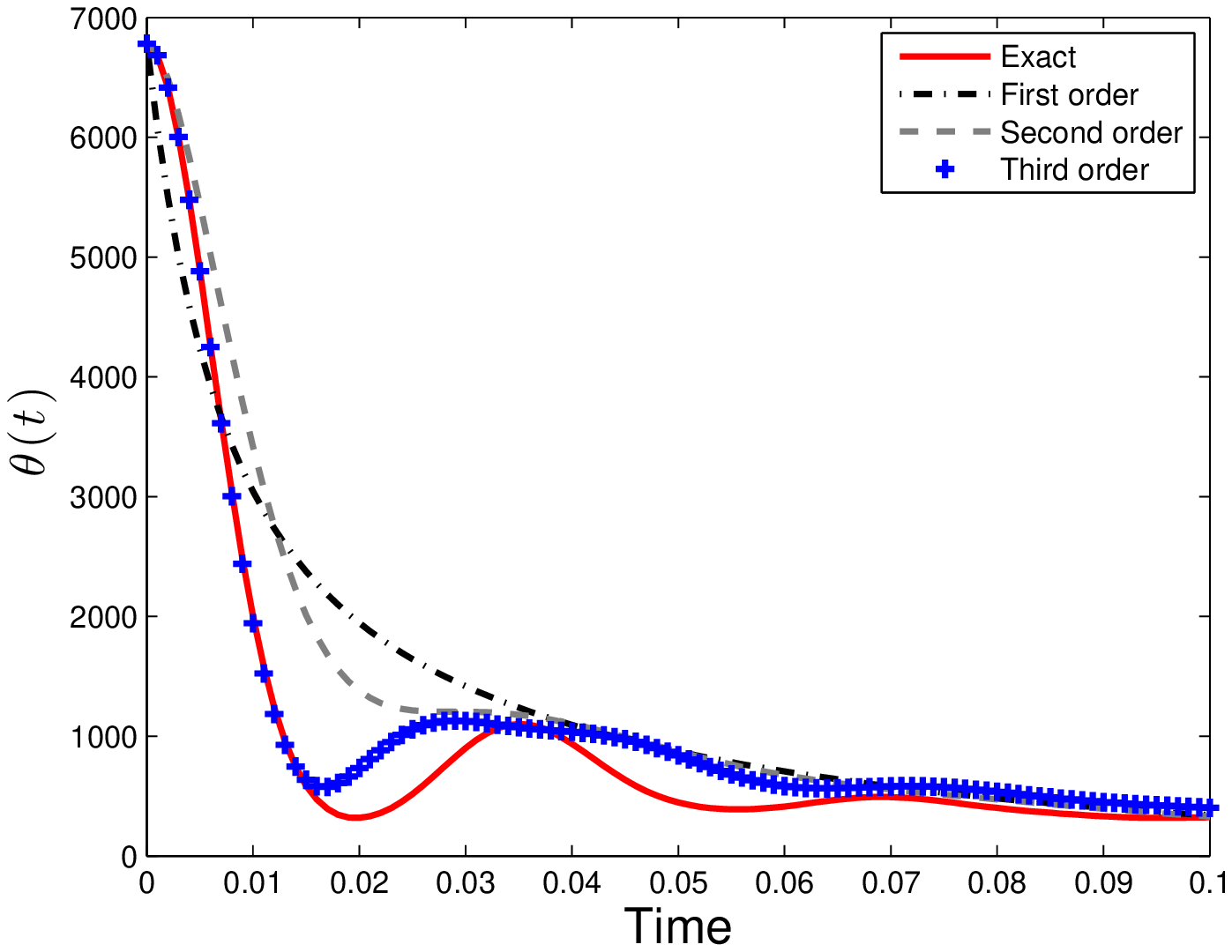}
\includegraphics[scale=.4]{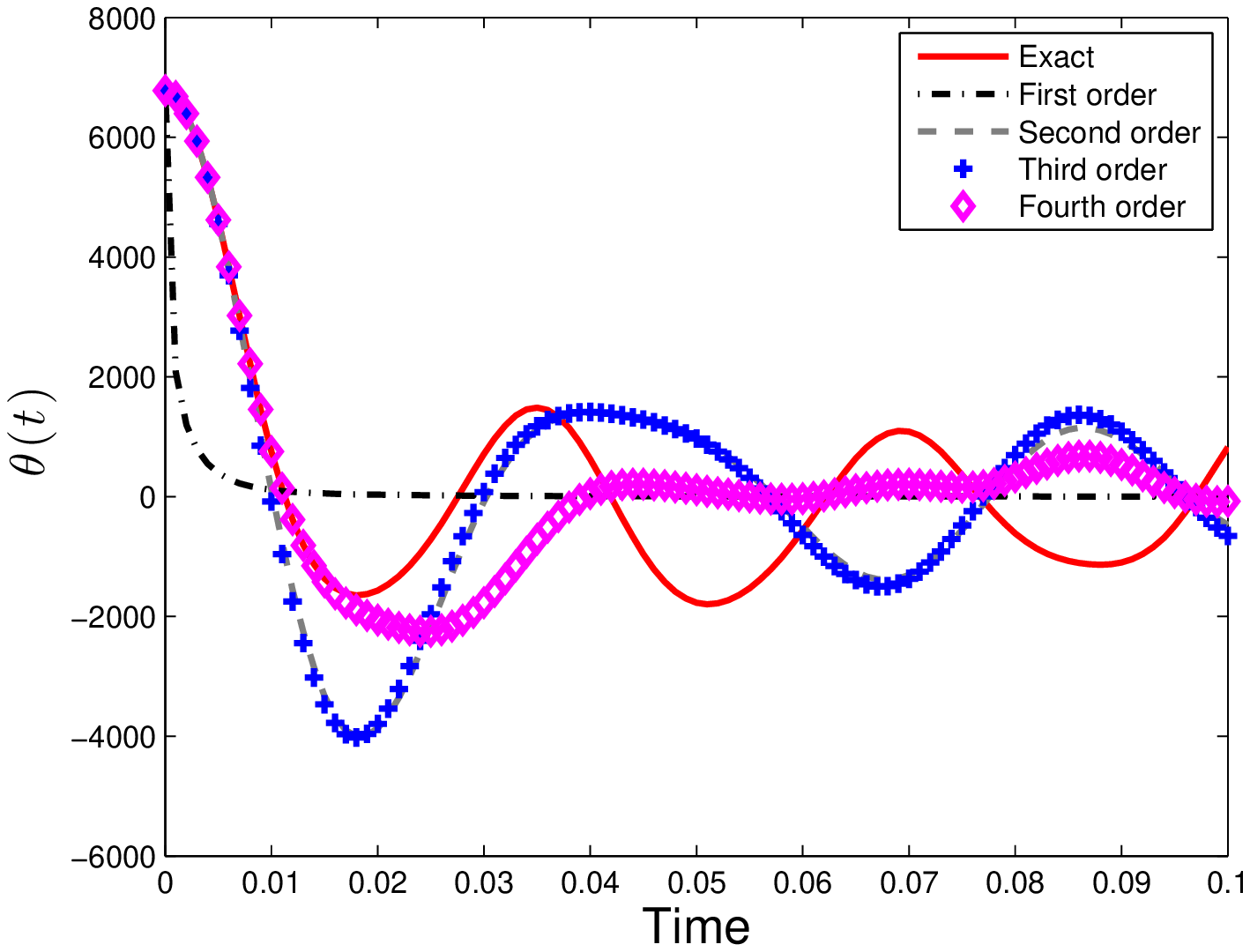}
\caption{Comparison between the exact kernel function \eqref{eq: theta} and approximations from the first, second and third order approximations. Plotted is the diagonal entry of the kernel function, $\theta_{66}(t)$ corresponding to {\red the last rotation component of the first residue.} The solid plot indicates the exact kernel function as in equation \eqref{eq: theta}, the dashed-dot and dashed lines are respectively for the first order and second order approximations, and $+$ represents the results from the third order approximation {\red and diamond corresponds to the forth order approximation in the figure on the right.} Left: $\Gamma=91 ps^{-1}$; right: $\Gamma=5 ps^{-1}$. The time scale is in pico seconds.}\label{fig: memker}
\end{figure}

Next, {\blue in Figure \ref{gm91}}, we show a comparison among the velocity auto-correlations for the case $\Gamma=91 ps^{-1}$, which is the default value in the molecular simulation package TINKER. {\blue Interested readers are referred to Appendix E for the details on the computation of the auto-correlation.}
In this case, all the time correlation functions exhibit exponential decay, indicating that the dynamics is over  damped. {\blue The correlation is already close to zero around time $t=0.1 ps$.}
In this case, the zeroth-order method gives poor result. But the results from the other three methods are in excellent agreement with the exact result.  The second and third order methods have slightly better accuracy. 
\begin{figure}[htp]
\includegraphics[scale=.5]{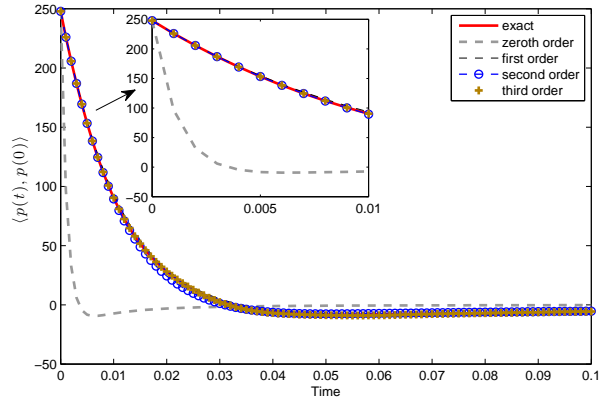}
\caption{Comparison of the velocity autocorrelation  from the exact and approximate models for $\Gamma=91 ps^{-1}$. Plots are for the {\red last rotation component} of the first residue. The time scale is in pico seconds. }\label{gm91}
\end{figure}

\medskip

{\blue  Following the previous experiment,  } we repeat the computation with damping coefficient $\Gamma=5 ps^{-1}$, and  the results are shown  in Figure \ref{gm5}.
In this case, the time correlation functions start to shown oscillatory patterns, {\Blue indicating that the memory effect is much stronger.} {\blue In light of the slow decay, we present results for a longer time period compared to the over-damped case}. Again, we see that the zeroth order approximation gives poor results, while the first-order method give is slightly better. Meanwhile, the second  and third order methods provide significant improvement around $t=0$. The inset figure shows a close-up view of the resulting correlation functions near $t=0.$ 
\begin{figure}[htp]
\includegraphics[scale=.5]{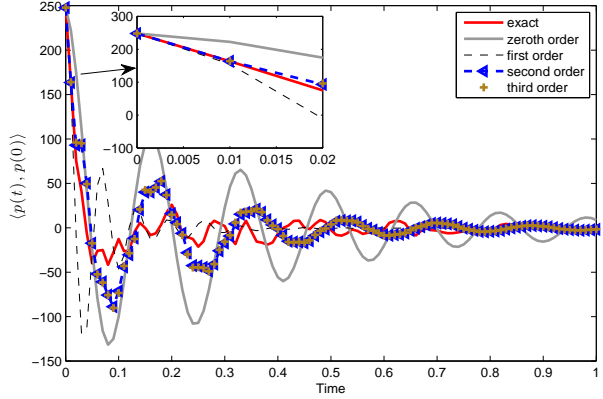}
\caption{Comparison of the velocity autocorrelation from the exact and approximate models with $\Gamma=5 ps^{-1}.$ Plots are for the last rotation component of the first residue.  The time scale is in pico seconds. The inset figure shows a close-up view of the resulting correlation functions near $t=0.$ 
}\label{gm5}
\end{figure}
\bigskip

\section{Conclusion}

This paper presented a derivation of a coarse-grained model from the full Langevin dynamics. The derivation has been focused on the resulting random noise, memory effect,  and the fluctuation-dissipation theorem, which is a necessary condition for the coarse-grained model to have the correct equilibrium statistics. Our main finding is a
 generalization of the generalized Langevin dynamics, together with a combined form of the first and second fluctuation-dissipation theorem. 

In the second part of the paper, a systematic approach to approximate the memory term was illustrated. The novel 
aspect is a rational approximation in the Laplace domain, which in the time domain, corresponds to an extended system with {\it no memory}.  This significantly reduces the computational cost. Furthermore, it has been shown that the random noise term in the generalized Langevin equation can be approximated indirectly by introducing white noises in the extended system. More importantly, the fluctuation-dissipation theorem still holds at each level of approximations. This is a property that has not been emphasized in other approximation methods, e.g., \cite{tuckerman1991stochastic,lange2006collective,IzVo06,berkowitz1981memory,fricks2009time,li2014construction,Darve_PNAS_2009}.

The current approach can be extended/improved in several directions. First, 
 a Hermite type of interpolation has been used in the approximation of the Laplace transform of memory function, and the interpolation is done at $\lambda=0$ and $\lambda=+\infty (s=0).$  It is clear that one can introduce other data points or interpolation methods to enhance the accuracy of the approximation.  {\blue As a demonstration, we did a simple test simulation (results shown  in Figure \ref{gm5a}) using the same interpolation points at $\lambda=0$ and $\lambda=+\infty$ but with different order of derivatives involved. In short, for the second order scheme $R_{2,2}$, we determine the four coefficients in the rational function as follows: We matched first and second derivatives at $\lambda=0$, and zeroth and first derivatives at $\lambda= +\infty$ (or $s=0$). For the third order scheme, for the two additional coefficients, we matched the third derivative at $\lambda=0$ and second derivative at $\lambda=+\infty$. The results are overall more satisfactory than our previous choices, indicating that there is a lot of flexibility in choosing the matching conditions. } {\Blue This approach would be more useful for the cases where the memory effect is much stronger.}
 \begin{figure}
 \includegraphics[scale=.5]{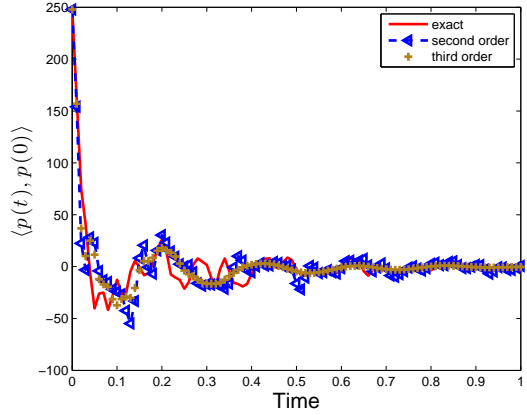}
 \caption{Comparison of the velocity autocorrelation from the exact and approximate models ($R_{2,2}$ and $R_{3,3}$) with $\Gamma=5 ps^{-1}.$ The matching conditions for $R_{2,2}$ involve two conditions at each end, and the conditions for $R_{3,3}$ contain three conditions at each end.  Plots are for the last rotation component of the first residue.  The time scale is in pico seconds. }\label{gm5a}
 \end{figure}

Secondly, we have only tested the methods for the case when the mean force is linear, e.g., an elastic network type of model. In this case, explicit forms for all the solutions are available, so that direct comparison can be made without the influence of the numerical error. It would be of great practical importance to test problems with a more realistic potential of mean forces, e.g., the ones obtained from the coarse-grained force field \cite{noid_multiscale_2008,monticelli2008martini}.  

In this paper, we have based our approximation on the moments $M_\ell$ \eqref{eq: moms}, which can be extracted from  the spectra of the molecular structure (the matrix $A$) and the damping coefficient ($\Gamma$). We would like to mention a data-driven approach, which makes use of the time series of the coarse-grain variables, and formulate the problem as an inverse problem. For instance, the Kalman filter technique has been used in \cite{fricks2009time,harlim2015parametric} to estimate the parameters $B_0$ and $C_0$ in the first order model \eqref{eq: 1st-order}, and in \cite{lei2016generalized}, the moments $M_i$ are directly linked to the correlations of the CG variables, which in turn determine the coefficients $B_i$s and $C_i$s. In all these works, the rational approximation in terms of the Laplace transform has been crucial.  Which approach is more appropriate depends on the information available to the practitioners. 

{\Blue Another interesting scenario is when the GLE is used to model subdiffusive behavior. One well-known example is where the kernel function obeys a power law \cite{kou2008stochastic}. In this case, we anticipate the current methodology to be useful up to certain time scale. When the long-time sub-diffusive behavior is of interest, the method certainly has to be modified. For example, when the kernel takes the form of $t^\alpha$, the Laplace transform will exhibits a singularity at the origin. Meanwhile, the current rational approximating function approaches to a finite value, and therefore the form of the rational function has to be modified accordingly in order to take into account the singularity. This would be an interesting line of work for us to pursue further.}

Finally, it is possible for the kernel function to depend on the current state of the coarse-grain variables, which means that they have to be continuously updated. These issues are important for the application to protein simulations that involve conformational changes, and they will be considered in separate works. 

The derivations presented in this paper, along with the calculations of the velocity correlation functions, involve some important, but lengthy mathematical manipulations. We included the details in the Appendix for interested readers.

\begin{acknowledgements}
This research was supported by NSF under grant DMS-1412005, DMS-1216938 and DMS-1619661. 
\end{acknowledgements}

\bigskip

\appendix

\numberwithin{equation}{section}

\section{Time correlation for linear Langevin models}

For linear Langevin dynamics, the velocity auto-correlation can be computed explicitly. 
This section illustrates the calculations.

Suppose that we have a linear Langevin dynamics model,
\begin{equation}\label{eq: lin'}
  \ddot{ u}= -A  u - \Gamma \dot{u} + W. 
\end{equation}

We may write it {\blue into a first order system} as follows,
\begin{equation}\label{eq: lin}
  \dot{ w} = D w + \Sigma {\blue \mu(t)}, 
\end{equation}
in which,
\begin{equation}
{\blue w=\left(\begin{array}{c}u \\p\end{array}\right),\quad
\mu=\left(\begin{array}{c}0 \\W\end{array}\right),\quad}
  D=
\left[
\begin{array}{cc}
  0 \;\;& I    \\
  -A \;\;& -\Gamma I    
\end{array}
\right], \quad \Sigma=
\left[
\begin{array}{cc}
  0 & 0    \\
  0 &\;\; \;\;2k_B T \Gamma I    
\end{array}
\right].
\end{equation}

For the linear Langevin dynamics, the equilibrium probability density is given by,
\begin{equation}
 \rho \sim e^{-\beta H}, \quad H= \frac12 u^T A u + \frac12 p^2,
\end{equation}
with $p=\dot{u}$ and $\beta=\frac{1}{k_B T}.$

Therefore, the covariance of the solution $w$ is given by,
\begin{equation}
Q= k_B T
\left[
\begin{array}{cc}
  A^{-1} \;\;& 0    \\
  0 &  I    
\end{array}
\right].
\end{equation}

Notice that $DQ + QD^T = -\Sigma$. This is known as the Lyapunov equation. In particular, when $w(0)$ is Gaussian with covariance $Q$, $w(t)$ is a stationary Gaussian process with time correlation given by, 
\begin{equation}
 \big\langle  w(t)  w(0)^T \big\rangle=k_B T e^{tD}Q.
\end{equation}
This formula will be used in many of our calculations.


{\blue Applying this formula to the full model, we find the time correlation of the coarse-grained momentum $p$,}
\begin{equation}
  \big\langle  p(t)  p(0)^T \big\rangle=k_B T  [ 0,\, \Phi^T] e^{tD}Q
\left[
\begin{array}{c}
  0    \\
  \Phi   
\end{array}
\right].
\end{equation}

\medskip
\section{The derivation of the GLE}
\subsection{Derivation of the memory kernel}\label{memdef}
We start with the last two equations in \eqref{eq: lg'}. To begin with, we recall the matrix $G$, defined in \eqref{eq: Gmat}.
Notice that the matrix can be factorized as follows,
\begin{equation}\label{eq: g-fac1}
 G=\left[\begin{array}{cc} 0 & I \\ I &\;\;\;\;-\Gamma_{22}\end{array}\right]\left[\begin{array}{cc} A_{22}\;\; & 0 \\ 0 &-I\end{array}\right].
\end{equation}
or,
\begin{equation}\label{eq: g-fac2}
 G=\left[
 \begin{array}{cc} 0 & -I \\ 
 I &\;\;\;\;\Gamma_{22}\end{array}
 \right]
 \left[
 \begin{array}{cc} 
 A_{22} \;\;\;\;& 0 \\
  0 &I
  \end{array}
  \right].
\end{equation}
It is also useful to have the inverse of $G$, given by,
\begin{equation}
G^{-1}=
\left[
\begin{array}{cc}
 A_{22}^{-1}\Gamma_{22} &  \;\;\;\;A_{22}^{-1}\\
 -I & 0
 \end{array}
 \right].
\end{equation}

Now the last two equations in \eqref{eq: lg'} can be expressed explicitly as, 
\[
\left[
\begin{array}{c}
 \xi(t)\\
\eta(t)
 \end{array}
 \right]= e^{-G t} 
 \left[
\begin{array}{c}
 \xi(0)\\
\eta(0)
 \end{array}
 \right]
 - \int_0^t e^{-G(t-s)}
 \left[
\begin{array}{c}
  0\\
 A_{21} q(s) + \Gamma_{21} p(s) 
 \end{array}
 \right] ds
 + 
 \int_0^t e^{-G(t-s)} \left[
\begin{array}{c}
  0 \\
{\blue f_2(s)}
 \end{array}
 \right] ds. \]

We take part of the memory term, and integrate by parts: 
\begin{align*}
&\int_0^t e^{-G(t-s)}
\left[\begin{array}{c} 0\\ A_{21} q(s)\end{array} \right]ds = 
e^{-G(t-s)}G^{-1} 
\left[\begin{array}{c} 0\\ A_{21} q(s) \end{array} \right] \Big|_0^t
- \int_0^t e^{-G(t-s)}G^{-1}
\left[\begin{array}{c} 0\\A_{21}p(s) \end{array} \right]ds\\
=&\left[ \begin{array}{c} A_{22}^{-1} A_{21} q(t) \\ 0 \end{array} \right]
-e^{-Gt} 
\left[ \begin{array}{c} A_{22}^{-1} A_{21} q(0) \\ 0 \end{array} \right]
-\int_0^t e^{-G(t-s)}\left[ \begin{array}{c} A_{22}^{-1} A_{21} p(s) \\ 0 \end{array} \right] ds
\end{align*}

Combining this with the remaining term in the  memory  integral, we have,
 \begin{equation}
 \begin{aligned}
 - \int_0^t e^{-G(t-s)}
 \left[
\begin{array}{c}
  0\\
 A_{21} q(s) + \Gamma_{21}p(s) 
 \end{array}
 \right] ds
 =& \int_0^t e^{-G(t-s)} 
 \left[
\begin{array}{c}
  A_{22}^{-1} A_{21}  \\
  -\Gamma_{21} 
 \end{array}
 \right] p(s)ds \\
 + &e^{-Gt} 
 \left[
\begin{array}{c}
  A_{22}^{-1} A_{21}  \\
 0
 \end{array}
\right] q(0)
- 
\left[
\begin{array}{c}
  A_{22}^{-1} A_{21}  \\
0
 \end{array}
\right] q(t).
\end{aligned}
\end{equation}
 
In the next step, we will substitute $\displaystyle \left[ \begin{array}{c} \xi(t)\\ \eta(t) \end{array}\right]$ into the first two equations  in (\ref{eq: lg'}), to eliminate the additional degrees of freedom and derive an effective equation for $q(t) $ and $p(t)$.

For clarity, we introduce more notations,
 \begin{equation}
 F_\text{eff}(q)= \Phi^T F(\Phi q) -A_{12}A_{22}^{-1}A_{21}q, \quad  \widehat{\xi}= \xi + A_{22}^{-1}A_{21} q,
\end{equation}
and,
\begin{align}
 \theta(t)= 
 \big[ A_{12}, \;\Gamma_{12}\big] e^{-Gt} 
 \left[
\begin{array}{cc}
 A_{22}^{-1} \;\;&  0\\
 0 & -I
 \end{array}
 \right]
 \left[
\begin{array}{c}
  A_{21}  \\
   \Gamma_{21}
 \end{array}
\right]. 
\end{align}

Collecting terms, we find that,
\begin{equation}
\left\{
\begin{aligned}
 \dot{q}= &  p, \\
   \dot{p}=& F_\text{eff}(q) - \Gamma_{11} p - \int_0^t \theta(t-s)  p(s)ds + \widehat{f}.
\end{aligned}
\right.
\end{equation}

This is a generalized Langevin equation with an additional damping, in the form of a memory term. The new random force is given by,
\begin{equation}\label{eq: f-ran}
\widehat{f} = {\blue f_1(t)}
  - \big[ A_{12}, \;\Gamma_{12}\big]  
 \int_0^t e^{-G(t-s)} \left[
\begin{array}{c}
  0 \\
{\blue f_2(s)}
 \end{array}
 \right] ds
 - \big[ A_{12}, \;\Gamma_{12}\big]
e^{-G t} 
 \left[
\begin{array}{c}
 \widehat{\xi}(0)\\
\eta(0)
 \end{array}
 \right]
. 
\end{equation}

\subsection{The fluctuation-dissipation theorem}\label{fluc}
Here we look at the random noise term and see how it is related to the damping coefficients.
Let the three terms in \eqref{eq: f-ran} be $f_1$, $\blue \widehat f_2$ and $f_3$, respectively.
One can see directly that,
\begin{equation}
 \bl f_1(t) f_1(t')^T \br = 2k_B T\, \Gamma_{11} \delta(t-t').
 \end{equation}

For $\blue f_3(t)$, we have,
\begin{equation}\label{eq: f2f2}
\begin{aligned}
\bl {\blue  f_3(t) f_3(t')}^T\br &= [A_{12}, \Gamma_{12} ] e^{-Gt} 
\left[ \begin{array}{cc} \bl \widehat{\xi}(0) \widehat{\xi}(0)^T \br &\bl \widehat{\xi}(0)\eta(0)^T \br \\ \bl \eta(0) \widehat{\xi}(0)^T \br &\bl \eta(0)\eta(0)^T\br \end{array}\right] 
e^{-G^T t'} \left[ \begin{array}{c} A_{21}\\\Gamma_{21} \end{array}\right]\\
&=k_B T[A_{12}, \Gamma_{12} ] e^{-Gt} \left[ \begin{array}{cc} A_{22}^{-1} \;\;\;\;&0\\ 0 &I\end{array}\right] 
e^{-G^T t'} \left[ \begin{array}{c} A_{21}\\\Gamma_{21} \end{array}\right]
\end{aligned}
\end{equation}

We now consider $\blue \widehat f_2(t).$ Assume that $t'\le t,$ we have, 
\[
\begin{aligned}
\bl{\blue \widehat f_2}(t){\blue \widehat f_2}(t')^T\br&=k_BT \big[ A_{12}, \;\Gamma_{12}\big]  
 \int_0^{t'} e^{-G(t-s')}
  \left[
\begin{array}{cc}
  0 & 0 \\
  0 &\;\;\;\;2 \Gamma_{22} 
   \end{array} \right] 
e^{-G^T(t'-s')}
ds'  \left[ \begin{array}{c} A_{21}\\\Gamma_{21} \end{array}\right]
\end{aligned}
\]

We notice that 
\[
  G \left[ \begin{array}{cc} A_{22}^{-1} \;\;\;\;&0\\ 0 &I\end{array}\right] +  \left[ \begin{array}{cc} A_{22}^{-1} \;\;\;\;&0\\ 0 &I\end{array}\right] G^T=  \left[\begin{array}{cc}0 & 0 \\0 & \;\;2\Gamma_{22} \end{array}\right].\]
Therefore, this integral can be simplified to,
\begin{equation}\label{eq: f3f3}
\begin{aligned}
\bl {\blue \widehat f_2}(t){\blue \widehat f_2}(t')^T\br&=k_BT   [A_{12}, \Gamma_{12} ] 
 e^{-G(t-t')} \left[ \begin{array}{cc} A_{22}^{-1} \;\;\;\;&0\\ 0 &I\end{array}\right]  \left[ \begin{array}{c} A_{21}\\
 \Gamma_{21} \end{array}\right] \\
  &- k_BT  [A_{12}, \Gamma_{12} ] 
  e^{-Gt} \left[ \begin{array}{cc} A_{22}^{-1} \;\;\;\;&0\\ 0 &I\end{array}\right] 
e^{-G^T t'} \left[ \begin{array}{c} A_{21}\\\Gamma_{21} \end{array}\right]
\end{aligned}
\end{equation}

{ The second term will  be cancelled by $\bl{\blue \widehat f_2}(t){\blue \widehat f_2}(t')^T \br$. But the first term is slightly different from the memory function $\theta(t).$ In particular, the matrix in the middle has an entry $I$ instead of $-I.$}

To complete the derivation, we have to compute the cross terms $\bl f_1(t) {\blue \widehat f_2}(t')^T \br$
and $\bl {\blue \widehat f_2}(t) f_1(t')^T \br$. It is straightforward to show that  $\bl f_1(t) {\blue \widehat f_2}(t')^T \br=0.$
For the other term, we have,
\begin{equation}
\begin{aligned}
 \bl {\blue \widehat f_2}(t) f_1(t')^T \br=& - 2k_B T\big[A_{12}, \Gamma_{12} \big] e^{-G(t-t')} 
 \left[ \begin{array}{c}  0\\ \Gamma_{21} \end{array}\right]\\
 =&  k_BT   [A_{12}, \Gamma_{12} ] 
 e^{-G(t-t')} \left[ \begin{array}{cc} 0 &0\\ 0 &\;\;\;\;-2 I\end{array}\right]  \left[ \begin{array}{c} A_{21}\\\Gamma_{21} \end{array}\right] .
 \end{aligned}
\end{equation}
This term can be combined with the first term in \eqref{eq: f3f3}, and it gives $\theta(t)$.

This proves the fluctuation-dissipation theorem:
\begin{equation}
 \bl \widehat{f}(t) \widehat{f}(t')^T \br = 2k_B T \Gamma_1 \delta(t-t') + k_B T \theta(t-t').
\end{equation}

\bigskip
{\Blue
A natural extension of the current framework is to Langevin dynamics models, in which the damping coefficient is depends on the position of the particles. For instance, in the dissipative particle dynamics (DPD) models \cite{espanol1995statistical,hoogerbrugge1992simulating}, they are expressed as functions of the inter-particle distances. An immediate observation is that $\Gamma$ depends on the current time, the stochastic model will have variable coefficients.  In this case,  we define the matrix $G$ as in \eqref{eq: g-fac1}, but write it as $G(t)$ to show the time-dependence.   To facilitate the derivation, we introduce the fundamental matrix, defined by the ODEs
\begin{equation}
 \frac{\partial}{\partial t} Y(t,s) = -G(t) Y(t,s), Y(s,s)=I.
\end{equation}
It also satisfies the equation, $ \frac{\partial}{\partial s} Y(t,s) =  Y(t,s) G(s).$

With the fundamental matrix, we can write the solution of the last two equations in \eqref{eq: lg'}   as follows,
\begin{equation}
\begin{aligned}
 \left[\begin{array}{c}\xi(t) \\ \eta(t)\end{array}\right]=& Y(t,0)   \left[\begin{array}{c}\xi(0) \\ \eta(0)\end{array}\right]
 + \int_0^t Y(t,s) \left[\begin{array}{c}0 \\ \sigma w(s) \end{array}\right] ds\\
 &+ \int_0^t Y(t,s) \left[\begin{array}{c}0 \\ -A_{21} q(s) - \Gamma_{21}(s) p(s) \end{array}\right] ds.
\end{aligned}
\end{equation}
Here, to demonstrate the ideas more easily, we have omitted the pair-wise form of the damping coefficients in DPD and simply wrote it in a matrix form. 

The remaining steps are the same as the derivation in the previous section. In particular, the memory term becomes,
\begin{equation}
- \int \theta(t,t') p(t') dt', \quad \text{with} \; \theta(t,t')= [A_{21} \;\Gamma_{21}(t) ] Y(t,t') \left[\begin{array}{c}A_{22}^{-1} A_{12} \\ -\Gamma_{21}(t') \end{array}\right].
\end{equation}

The random noise  is still a  Gaussian process, having time correlation,
\begin{equation}
  \bl \widehat{f}(t) \widehat{f}(t')^T \br = 2k_B T \Gamma_{11}(t) \delta(t-t') + k_B T  \theta(t,t'). 
\end{equation}

 The main observation here is that the noise is no long a stationary process, since the correlation can not be written as a function of $t-t',$ and the memory kernel is no longer a convolution.   }

\section{Properties of the memory kernel}

We can show that this matrix is symmetric. 
\begin{align*}
&\theta(t)= \big[ A_{12}, \;\Gamma_{12}\big] \sum_n \frac{t^n}{n!} G^n
 \left[
\begin{array}{cc}
 A_{22}^{-1} \;\;\;\;&  0\\
 0 &- I
 \end{array}
 \right]
 \left[
\begin{array}{c}
  A_{21}  \\
   \Gamma_{21}
 \end{array}
\right]. \\
=& \sum_n \frac{t^n}{n!} \big[ A_{12}, \;\Gamma_{12}\big]  \left[\begin{array}{cc} 0 & I \\ I &-\Gamma_2\end{array}\right]\left[\begin{array}{cc} A_{22} \;\;\;\;& 0 \\ 0 &-I\end{array}\right] \cdots\left[\begin{array}{cc} 0 & I \\ I &-\Gamma_2\end{array}\right]\left[\begin{array}{cc} A_{22} \;\;\;\;& 0 \\ 0 &-I\end{array}\right]
 \left[
\begin{array}{cc}
 A_{22}^{-1} \;\;\;\;&  0\\
 0 & -I
 \end{array}
 \right]
 \left[
\begin{array}{c}
  A_{21}  \\
  \Gamma_{21}
 \end{array}
\right]. 
\end{align*}
\begin{align*}
\theta^T(t) = \sum_n \frac{t^n}{n!} \big[ A_{12}, \;\Gamma_{12}\big]  \left[\begin{array}{cc} 0 & I \\ I &-\Gamma_2\end{array}\right]\left[\begin{array}{cc} A_{22} \;\;& 0 \\ 0 &-I\end{array}\right] \cdots\left[\begin{array}{cc} 0 & I \\ I &-\Gamma_2\end{array}\right]
 \left[
\begin{array}{c}
  A_{21}  \\
 \Gamma_{21}
 \end{array}
\right] = \theta(t)
\end{align*}

In addition, we see that,
\begin{equation}
\theta(0)=A_{12} A_{22}^{-1} A_{21}- \Gamma_{12} M_2^{-1} \Gamma_{21}. 
\end{equation}

Finally,
\begin{equation}\label{memoryk}
\int_0^{\infty} \theta(t) dt = 
 \big[ A_{12}, \;\Gamma_{12}\big] G^{-1} 
 \left[
\begin{array}{cc}
 A_{22}^{-1}\;\;\;\; &  0\\
 0 & -I
 \end{array}
 \right]
 \left[
\begin{array}{c}
  A_{21}  \\
  \Gamma_{21}
 \end{array}
\right]. 
\end{equation}

\section{The proof of Theorem 2}\label{sec: 2nd}
Using the form of the rational function $R_{2,2}$ \eqref{eq: r22} and the properties of Laplace transform, we can write down a differential equation for the approximate memory kernel,
\begin{equation}
\theta''= B_0 \theta'+ B_1 \theta,
\end{equation}
together with the initial conditions,
\begin{equation}
 \theta(0)=M_0, \quad \theta'(0)=M_1
\end{equation}
which are drawn from the interpolation conditions \eqref{eq: m2}.

By defining $\theta_1=\theta'-B_0\theta,$ we can write this in a first order form,
\begin{equation}
\begin{aligned}
& \theta'_1 =  B_1 \theta, \\
 &\theta' =  B_0 \theta + \theta_1, \\
 &\theta(0) =  C_0, \quad \theta_1(0)=M_1 - B_0C_0.
\end{aligned}
\end{equation}
From the second matching conditions \eqref{eq: m2}, we find that $\theta_1(0)=C_1.$

As a result, the approximate memory kernel can be written in a matrix exponential form,
\begin{equation}\label{eq: theta2}
 \theta(t)= 
 [ 0 \quad I] e^{ t \wh{B}
 } 
   \left[\begin{array}{c}
     C_0 \\ 
     C_1 
     \end{array}\right], \quad \wh{B}= \left[\begin{array}{cc}
   0\quad &\; B_1 \\
   I\quad &\; B_0 \\
   \end{array}\right].
\end{equation}

We will derive the initial covariance for the second order approximation (\ref{eq: 2nd-order}). 
Consider the linear system as in Appendix \eqref{eq: lin} for {\red $u=(p,z_1,z)^T$}. In particular, we have,
\begin{align*}
D=\left[\begin{array}{ccc} -\Gamma_{11} \;\;\;\;& 0\;\;\;\; & -I \\{\blue C}_1 & 0 \;\;\;\;& B_1 \\
{\blue C}_0 & I \;\;\;\;& B_0\end{array}\right].
\end{align*}
Let us choose the initial condition for $u$ as Gaussian with mean zero and covariance,
\begin{align*}
Q=\left[\begin{array}{ccc}I \;\;\;\;& 0\;\; & 0 \\0\;\;\;\; & Q_1\;\; & Q_{12} \\
0\;\;\;\;& Q_{12}^T\;\; & Q_2\end{array}\right],
\end{align*}
then
\begin{align*}
DQ=\left[\begin{array}{ccc}-\Gamma_{11} & -Q_{12}^T & -Q_2 \\
{\blue C}_1 & B_1Q_{12}^T & B_1Q_2 \\
{\blue C}_0 &\;\; Q_1+B_0Q_{12}^T &\;\;\;\;Q_{12}+B_0Q_2\end{array}\right].
\end{align*}

We seek a simple case when $DQ$ is an asymmetric matrix, which leads to the choices,
\begin{equation}\label{eq: Q1-Q2}
Q_1=-{\blue C}_0^TB_1^T-B_0{\blue C}_1^T, \quad \\
 \quad  Q_2={\blue C}_0, \quad \\
\quad  Q_{12}={\blue C}_1.
\end{equation}

In light of the Lyapunov equation for stochastic differential equations \cite{risken1984fokker}, this gives the covariance matrix for the random noise $\blue \zeta(t)$ and $\zeta_1(t)$ in the second order model. More importantly, the resulting solution will become a stationary Gaussian process thanks to the Lyapunov condition.

{\red With the initial covariance and the covariance of the noise $(\zeta_1(t), \zeta(t))$ determined, we can solve the two equations for $z_1(t)$ and $z(t)$, and substitute it back to the second equation (\ref{eq: 2nd-order}). Similar to the proof of Theorem 1, we find three terms, 
  \begin{equation}
 z(t) = [0\quad I] \left\{ e^{tD}  \left[\begin{array}{c}
     z_1(0) \\ 
     z(0)
     \end{array}\right]
       + 
        \int_0^t e^{(t-\tau) D}  \left[\begin{array}{c}
     \zeta_1 (\tau) \\ 
     \zeta(\tau) 
     \end{array}\right] d
     \tau + 
       \int_0^t e^{(t-\tau) D}  \left[\begin{array}{c}
     C_1 \\ 
     C_0 
     \end{array}\right] p(\tau) d
     \tau \right\}.
\end{equation}

We immediately see that the last term gives rise to an approximation to the memory term, with memory kernel exactly given by \eqref{eq: theta2}, which as explained at the beginning of this section, correspond to the rational approximation of the Laplace transform \eqref{eq: r22}. In addition, the first two terms form a stationary Gaussian process, denoted by $g(t)$, since the Lyapunov condition has been imposed. This $g(t)$ will lead to an approximation of the colored noise $\wh{f}(t)$ in the CG model \eqref{eq: GLE}. In particular, the time correlation of this process is given by, 
\begin{equation}
\big\langle g(t) g(t') \big\rangle=  k_B T [0\quad  I] e^{(t-t')D} Q \left[\begin{array}{c}
     0  \\ 
     I 
     \end{array}\right].
\end{equation}
From \eqref{eq: Q1-Q2}, we find that,
\[Q \left[\begin{array}{c}
     0  \\ 
     I 
     \end{array}\right] = 
      \left[\begin{array}{c}
     C_0  \\ 
     C_1
     \end{array}\right],\]
which implies that,
\[\big\langle g(t) g(t') \big\rangle=  k_B T \theta_2(t-t'),\]
proving the consistency.

}


\section{The derivation of the  time correlation for the approximations to the GLE}\label{auto}
{\blue We start with the general Langevin equations,  written as,
\begin{align*}
&\dot{q}=p,\\
&\dot{p}=-{\blue A}q-\Gamma p - \int_0^t \theta(t-\tau) p(\tau) d\tau + {\blue f(t)}. 
\end{align*}}

Assume that the noise term is independent of $p(0)$. We define 
\begin{align*}
&D(t)=\langle q(t), q(0)^T\rangle,\\
&{\blue H}(t)=\langle q(t), p(0)^T \rangle,\\
&{\blue E}(t)=\langle p(t), p(0)^T \rangle.
\end{align*}
By multiplying the GLE by $q(0)$ and $p(0)$ and taking averages, one can derive the following equations for the correlation functions:
\begin{align*}
&\dot{D} = - {\blue H},\\
&\dot{{\blue H}} = {\blue E},\\
&\dot{\blue E} = - {\blue A H}  - \Gamma {\blue E} - \int_0^t \theta(t-\tau) {\blue E}(\tau) d\tau.
\end{align*}

We now define the memory term in this system ${\blue Z}=\int_0^t \theta(t-\tau) {\blue E}(\tau) d\tau$, and similar to our derivation of the first order approximation to the memory kernel function, we find that,
\begin{align*}
\blue
\dot{Z}=B_0 Z + C_0 E, \quad Z(0)=0.
\end{align*}

Then, the system for the correlation function of the first order approximation  becomes:
{\blue
\begin{align*}
&\dot{D} = - H,\\
&\dot{H} = E,\\
&\dot{E} = - {A} H - \Gamma E - Z,\\
&\dot{Z}=B_0 Z + C_0 E,\\
&H(0)={\bf{0}},\quad E(0) = k_B T I,\quad Z(0) = {\bf 0}. 
\end{align*}
}
Similarly, the corresponding  equations for the second-order approximation are given by,
\begin{align*}
&\dot{D} = - H,\\
&\dot{H} = E,\\
&\dot{E} = - {A} H - \Gamma E - Z,\\
&\dot{Z}=Z_1+ B_0 Z + C_0 E,\\
&\dot{Z_1} = B_1 Z + C_1 E,\\
&H(0)={\bf{0}},\quad  E(0) = k_B T I, \quad  Z(0) = {\bf 0}, \quad Z_1(0)={\bf 0}. 
\end{align*}

And we can also derive the equations for the correlation functions from the third-order model,
{\blue
\begin{align*}
&\dot{D} = - H,\\
&\dot{H} = E,\\
&\dot{E} = - {A} H - \Gamma E - Z,\\
&\dot{Z}=Z_1+ B_0 Z + C_0 E,\\
&\dot{Z_1} = Z_2+ B_1 Z + C_1 E,\\
&\dot{Z_2}=B_2Z + C_2 E,\\
&H(0)={\bf{0}},\quad  E(0) = k_B T I,\quad  Z(0) = {\bf 0},\quad  Z_1(0)={\bf 0},\quad Z_2(0)={\bf 0}. 
\end{align*}
} 

Once we write these unknown quantities in the form of linear system of autonomous ordinary differential equations,
the solutions are readily available. In particular, they can expressed in  terms of the fundamental solutions, in the form of matrix exponential. We can then evaluate them directly using methods from numerical linear algebra.  

\bibliographystyle{plain}

\bibliography{mem0,GLERef1}
\end{document}